\documentclass[reqno]{amsart}

\usepackage{amssymb,latexsym}
\usepackage{amsfonts}
\usepackage{dsfont}
\usepackage[pdftex]{hyperref}
\usepackage{color}
\usepackage{graphicx}
\usepackage{parskip}
\usepackage{mathrsfs}
\usepackage{tasks}
\usepackage{enumerate}
\usepackage{amsmath}

\theoremstyle{plain}
\newtheorem{thm}{Theorem}[section]

\newtheorem{lemma}{Lemma}[section]

\numberwithin{equation}{section}

\newcommand{\p}{\partial}
\newcommand{\m}{\mathbb}
\newcommand{\om}{\omega}

\newcommand{\rr}{\mathbb{R}}

\newcommand{\zz}{\mathbb{Z}}

\newcommand{\ci}{\mathbb{T}}

\newcommand{\wh}{\widehat}

\newcommand{\vp}{\varphi}
\newcommand{\s}{\sigma}

\begin{document}

\title[Generalized CH2]{A Note on a Generalized Two Component \\ Camassa-Holm System in Sobolev Spaces}
\author{Ryan C. Thompson}
\date{June 12, 2024}
\address{Department of Mathematics\\University of North Georgia\\ Dahlonega, GA 30597}
\email{ryan.thompson@ung.edu}

\begin{abstract}
In this paper, we consider a generalized two component Camassa-Holm system.  Based on local well-posedness results and lifespan estimates, we establish sharpness of continuity on the data-to-solution map by showing that it is not uniformly continuous from product Sobolev spaces $H^s \times H^{s}$ to $C([0,T]; H^s \times H^{s})$.  The proof of nonuniform dependence is based upon approximate solutions.
\end{abstract}

\maketitle

%
%
%
%

\section{Introduction}
We consider the Cauchy problem for a generalized two-component Camassa-Holm system
\begin{equation}
\label{gench}
\begin{cases}
m_t+v^pm_x+av^{p-1}v_xm=0, \ \ \ t>0, \ x \in \rr \\
n_t+u^qn_x+bu^{q-1}u_xn=0, \ \ \ t>0, \ x\in \rr \\ 
u(x,0) = u_0(x), \ v(x,0)=v_0(x), \ \ \ t=0, \ x\in \rr
\end{cases}
\end{equation}
where $p,q \in \zz^+$, $a,b \in \rr$, $m=u-u_{xx}$, and $n = v-v_{xx}$ which was recently proposed by Pan, Zhou, and Qiao in \cite{pzq}.  We show that the data-to-solution map is not uniformly continuous in product Sobolev spaces $H^s \times H^s$ for $s>5/2$.  Zhou, Qiao, and Mu previously established nonuniform dependence on the initial data in \cite{zhou3} in the nonperiodic case.  In this paper, we consider the periodic case and prove a similar result which concludes the story of nonuniform dependence in the product Sobolev spaces.

System \eqref{gench} was proposed in order to generalize and unify multiple equations that generated solitary waves and exhibited the physical property of wave breaking.  Indeed, throughout the twentieth century there was marked progress in the field of water wave theory that generated a multitude of nonlinear wave equations that can be found from system \eqref{gench} by choosing parameters $p,q,a,b$ appropriately.

Considering the case when $u=v$, $a=b=2$, and $p=q=1$ yields the celebrated Camassa-Holm equation
\begin{equation}
\label{ch}
m_t+um_x+2mu_x=0,
\end{equation}
which appears in the context of hereditary symmetries studied by Fokas and Fuchssteiner \cite{ff}.  This equation was first written explicitly and derived from the Euler equations  by Camassa and Holm in \cite{ch}, where they also found its 
``peakon" traveling wave solutions.
These are solutions with a discontinuity in the first spatial derivative at its crest. The simplest one in the non-periodic case
is of the form $u_c(x, t)=ce^{-|x-t|}$, where $c$ is a positive constant.
The CH equation possesses many  other 
remarkable properties such as infinitely many conserved quantities, a bi-Hamiltonian structure and a Lax pair.
For more information about how CH arises in the context of hereditary symmetries we refer to \cite{ff}.  Concerning it's physical relevance, we refer the reader to the works by Johnson \cite{j1}, \cite{j2} and Constantine and Lannes \cite{cl}.

For $u=v$, $a=b=3$, and $p=q=1$, we obtain the Degasperi-Procesi (DP) equation
\begin{equation}
\label{dp}
m_t+um_x+3mu_x=0,
\end{equation}
which was discovered by Degasperi and Procesi \cite{dp} in 1998 as one of the three equations to satisfy asymptotic integrability conditions in the following family of equations
\begin{equation}
\label{gendp}
u_t+c_0u_x+\gamma u_{xxx}-\alpha^2u_{xxt} = (c_1u^2+c_2u_x^2+c_3uu_x)_x,
\end{equation}
where $\gamma, \alpha, c_0, c_1, c_2, c_3 \in \rr$ are constants.  The other integrable members of \eqref{gendp} are the CH and KdV equations.  Furthermore, we note that the DP equation also posesses peaked solitons of the form $u_c(x,t) = ce^{-|x-ct|}$and if we let $a=b \in \rr$, in this case, we obtain the $b$-family for which the CH and DP equations are the only integrable members.

For $u=v$, $a=b=3$, and $p=q=2$, we obtain the Novikov equation (NE)
\begin{equation}
\label{nov}
m_t+m_xu^2+3muu_x=0,
\end{equation}
which was discovered by Vladimir Novikov in \cite{nov} as he was investigating the integrability of Camassa-Holm type equations of the form
\begin{equation}
(1-\p_x^2)u_t = P(u,u_x,u_{xx}, \dots),
\end{equation}
where $P$ is a polynomial of $u$ and its derivatives.  Like CH and DP equations, NE also exhibits peaked solitary wave solutions of the form $u_c(x,t) = \sqrt ce^{-|x-ct|}$.

In recent years, the Camassa-Holm type equations have been generalized into integrable two component systems such as the CH2 equation,
\begin{equation}
\label{ch2}
\begin{cases}m_t+um_x+2mu_x+\s\rho\rho_x=0, \\  \rho_t+(u\rho)_x=0, \end{cases}
\end{equation}
where $\s=\pm 1$.  Here $u(x,t)$ represents the horizontal velocity of the fluid and $\rho(x,t)$ is the horizontal deviation of the surface from equilibrium.

Constantin and Ivanov \cite{const} demonstrated how the system \eqref{ch2} is derived from shallow water wave theory through various perturbations of the Green-Naghdi equations.  Additional physical relevance was provided under Chen, Liu and Zhang \cite{clz} when they connected CH2 with the time dependent Schr\"{o}dinger spectral problem. The 2-component system \eqref{ch2} was also shown to be completely integrable in Falqui \cite{faq} and \cite{shab} and its bi-Hamiltonian structure was subsequently determined.  We now also know the system can be identified with the first negative flow of the AKNS hierarchy and possesses the peakon and multi-kink solutions \cite{clz}.

System \eqref{gench}, however, was motivated by the work of Cotter, Holm, Ivanov, and Percival on the Cross-Coupled Camassa-Holm system (CCCH) in \cite{chip}.  This system is realized when we take $a=b=2$ and $p=q=1$.  In \cite{eik}, a geometrical interpretation of CCCH is given along with a large class of peakon equations.  Results regarding wave breaking, continuity and analyticity of the data-to-solution map, and persistence properties for the CCCH system can be found in \cite{hhi, xliu, zhou1, zhou2}.

Recently, well-posedness of system \eqref{gench} was established in the product Besov spaces $B_{l,r}^s \times B_{l,r}^{s}$ for $s>\max\{2+1/l,3-1/l\}$ in \cite{pzq}.  In this paper we seek to use this result corresponding to product Sobolev spaces $H^s \times H^s$ to prove nonuniform dependence in a peridoc setting.  This is achieved by the use of approximate solutions, a solution size estimate in product Sobolev spaces found in \cite{zhou3}, and by exploiting the nonlocal form of the system which can be written as
\begin{equation}
\label{nl}
\begin{cases}
u_t+v^pu_x+I_1(u,v) = 0, \\ 
v_t+u^qv_x+I_2(u,v) = 0,
\end{cases}
\end{equation}
where
\begin{align}
&I_1(u,v) = (1-\p_x^2)^{-1}\left[\frac ap(v^p)_xu+\frac{p-a}{p}(v^p)_xu_{xx}\right]+(1-\p_x^2)^{-1}\p_x\left((v^p)_xu_x\right), \\ 
&I_2(u,v) = (1-\p_x^2)^{-1}\left[\frac bq(u^q)_xv+\frac{q-b}{q}(u^q)_xv_{xx}\right]+(1-\p_x^2)^{-1}\p_x\left((u^q)_xv_x\right),
\end{align}
$(1-\p_x^2)^{-1}f = G*f$ where $G$ is the Green's function $G=\frac12e^{-|x|}$ and $\mathcal{F}[G*f]= \frac1{1+\xi^2} \wh f(\xi)$ for any test function $f$. More generally, for any real number $s$, we will denote the operator $\lambda ^s = (1-\p_x^2)^s$, and note that $\| u\|_{H^s} = \| \lambda^s u\|_{L^2}$.  For brevity of notation, we will write $H^s(\ci)=H^s$ and 
\[
\|(u,v)\|_s = \|u\|_{H^s}+\|v\|_{H^s}.
\]

The continuity of the data-to-solution map plays a crucial role in well-posedness theory and the sharpness in regularity of said map has been investigated by multiple authors.  Indeed, Himonas et. al. produced sharpness results for equations such as CH, Novikov, Degasperis-Procesi and generalized CH-type-equations \cite{hh1} - \cite{holmes} in Sobolev spaces $H^s$ by showing that the data to solution maps were not uniformly continuous from $H^s$ into $C([0,T);H^s)$.  This methodology was further extended to systems and similar results were produced for the CH2 system in \cite{rct}, the Euler-Poisson system in \cite{holmtig}, and R2CH by Yang in \cite{yang} and Weng et. al. in \cite{wzzwm} in the product Sobolev spaces $H^s \times H^{s-1}$.  We draw our inspiration from these authors and prove the following result.

\begin{thm}
\label{nonunif}
Assume that $s>5/2$. Then the data-to-solution map corresponding to the generalized two component Camassa-Holm system \eqref{gench} is not uniformly continuous from $H^s(\ci) \times H^s(\ci)$ into $C([0,T]; H^s(\ci) \times H^{s}(\ci))$.
\end{thm}

%
%
%
%
\section{Nonuniform Dependence in Sobolev Spaces on $\m{T}$}
\begin{thm}
If $s>5/2$ then the data-to-solution map for the 2-component generalized Camassa-Holm system defined by the Cauchy problem \eqref{nl} is not uniformly continuous from any bounded subset of $H^s \times H^{s}$ into $C([0,T];H^s \times H^{s})$.
\end{thm}
It suffices to show that there exists two sequences of solutions $(u_n,v_n)$ and $(w_n,y_n)$ in $C([0,T];H^s \times H^{s})$ such that
\begin{align*}
&\|(u_n,v_n)\|_s+\|(w_n,y_n)\|_s \lesssim 1 \\
&\lim_{n \rightarrow \infty} \|(u_n(0)-w_n(0),v_n(0)-y_n(0))\|_s=0,
\end{align*}
and
\begin{equation*}
\liminf_{n \rightarrow \infty} \|(u_n(t)-w_n(t),v_n(t)-y_n(t))\|_s \gtrsim |\sin t|, \ \ \ 0\leq t \leq T <1.
\end{equation*}
\textbf{Approximate 2-Component gen-CH solutions.}  We consider the approximate solutions of the form
\begin{align}
\label{as}
&u^{\omega,n}(x,t)=\omega n^{-\frac1q}+n^{-s}\cos (nx-\omega^p t) \nonumber \\
&v^{\omega,n}(x,t)=\omega n^{-\frac1p}+n^{-s}\cos (nx-\omega^q t).
\end{align}
where $n \in \m{Z}^+$ and $\omega=\pm 1$.  Substituting \eqref{as} into the 2-component CH system generates the following expressions for the errors $E$ and $F$:
\begin{align*}
&E\doteq\partial_tu^{\omega,n}+(v^{\omega,n})^p\partial_xu^{\omega,n}+I_1(u^{\omega,n},v^{\omega,n}) \\
&F\doteq\partial_tv^{\omega,n}+(u^{\omega,n})^q\partial_xv^{\omega,n}+I_2(u^{\omega,n},v^{\omega,n}).
\end{align*}
To estimate the $H^\s$ norm of these errors, we use the following fact.  For $\s \in \m{R}$, $n \in \m{Z}^+$ and $n>>1$, we have
\begin{equation}
\label{cosest}
\|\cos(nx-\alpha)\|_{H^\s} \approx n^\s, \ \ \alpha \in \m{R}.
\end{equation}
The relation also holds if cosine is replaced by sine.  Hence, for $s\geq 0$ and $n>>1$ we have
\begin{align*}
&\|u^{\omega,n}(t)\|_{H^\s}=\|\omega n^{-\frac1q}+n^{-s}\cos (nx-\omega^p t)\|_{H^\s} \lesssim n^{-\frac1q}+n^{-s+\s} \\ 
&\|v^{\omega,n}(t)\|_{H^\s}=\|\omega n^{-\frac1p}+n^{-s}\cos (nx-\omega^q t)\|_{H^\s} \lesssim n^{-\frac1p}+n^{-s+\s}.
\end{align*}

\begin{lemma}
\label{erest}
For $n>>1$ and $s>5/2$,
\begin{align}
\label{eest}
&\|E(t)\|_{H^\s} \lesssim n^{r_{s,\s}} \\
\label{fest}
&\|F(t)\|_{H^{\s}} \lesssim n^{j_{s,\s}}
\end{align}
where
\begin{equation*}
r_{s,\s} = \begin{cases} \frac1p-2s+2, \ \ \ \ \ \ \ \ \ \ \ \     s<\frac1q-\s+4 \\ \frac1p-\frac1q-s+\s-2, \ \ \ s>\frac1q-\s+4 \end{cases}
\end{equation*}
and
\begin{equation*}
j_{s,\s} = \begin{cases} \frac1q-2s+2, \ \ \ \ \ \ \ \ \ \ \ \     s<\frac1p-\s+4 \\ \frac1q-\frac1p-s+\s-2, \ \ \ s>\frac1p-\s+4. \end{cases}
\end{equation*}
\end{lemma}
\begin{proof}
First we note the following derivatives of our approximate solutions
\begin{align*}
&\p_tu^{\omega,n}(x,t)=\omega^p n^{-s}\sin(nx-\omega^p t) \\ 
&\p_xu^{\omega,n}(x,t)=-n^{-s+1}\sin(nx-\omega^p t) \\ 
&\p_x^2u^{\omega,n}(x,t)=-n^{-s+2}\cos(nx-\omega^p t) \\ 
&\p_tv^{\omega,n}(x,t)=\omega^q n^{-s}\sin(nx-\omega^q t) \\ 
&\p_xv^{\omega,n}(x,t)=-n^{-s+1}\sin(nx-\omega^q t) \\ 
&\p_x^2v^{\omega,n}(x,t)=-n^{-s+2}\cos(nx-\omega^q t).
\end{align*}
Let $E_B$ denote the Burgers term portion of the error $E$ and $E_{nl}$ denote the nonlocal portion.  We also note that by binomial theorem, we have 
\[
(v^{\om, n})^p = \sum_{k=0}^p \begin{pmatrix} p \\ k \end{pmatrix} (\om n^{-\frac1p})^{p-k}(n^{-s}\cos(nx-\om^qt))^k,
\]
in which the first term of this summation is $\om^pn^{-1}$.  We consider this first time multiplied by the derivative $\p_xu^{\om,n}$ and added to $\p_tu^{\om,n}$ to see that
\[
\p_tu^{\om,n}+(\om^pn^{-1})\p_xu^{\om, n} = \omega^p n^{-s}\sin(nx-\omega^p t)+\om^pn^{-1}(-n^{-s+1}\sin(nx-\omega^p t)) = 0.
\]
Therefore, in the Burgers term of the first component we are left with 
\begin{align}
E_B &= \left(\sum_{k=1}^p \begin{pmatrix} p \\ k \end{pmatrix} (\om n^{-\frac1p})^{p-k}(n^{-s}\cos(nx-\om^qt))^k\right)\left(-n^{-s+1}\sin(nx-\omega^p t)\right) \nonumber \\ 
& = -\sum_{k=1}^p \begin{pmatrix} p \\ k \end{pmatrix} \om^{p-k}n^{\frac kp-ks-s}\cos^k(nx-\om^qt)\sin(nx-\omega^p t).
\end{align}
Since we have that $n>>1$, we may just consider the first term as the others are negligible.  At $k=1$ in the summation we have that
\[
E_B = -p\om^{p-1}n^{\frac1p-2s}\cos(nx-\omega^q t)\sin(nx-\omega^p t)+(\text{negligible terms}).
\]
Using a well-known identity that
\[
\cos\theta\sin\vp = \frac12[\sin(\vp+\theta)+\sin(\vp-\theta)],
\]
we find that
\[
E_B = -\frac12\om^{p-1}n^{\frac1p-2s}[\sin(\vp+\theta)+\sin(\vp-\theta)]+(\text{negligible terms}),
\]
where
\[
\vp = nx-\om^pt \ \ \text{and} \ \ \theta = nx-\om^qt.
\]
Estimating the Burgers term $E_B$ in the $H^\s$ norm we find that
\begin{equation}
\label{eb}
\|E_B\|_{H^\s} \lesssim n^{\frac1p-2s+\s}.
\end{equation}
We now estimate the non-local term of the first component in the $H^\s$ norm.  We may break this down into two terms $E_{nl}^1$ and $E_{nl}^2$, where
\begin{align}
\label{enl12}
&E_{nl}^1 = (1-\p_x^2)^{-1}\left[(v^{\om,n})^{p-1}\p_xv^{\om,n}(au^{\om,n}+(p-a)\p_x^2u^{\om,n})\right] \nonumber \\ 
&E_{nl}^2 = p\p_x(1-\p_x^2)^{-1}\left[(v^{\om,n})^{p-1}\p_xv^{\om,n}\p_xu^{\om,n}\right].
\end{align}
Throughout the nonlocal term we see the presence of the term $(v^{\om,n})^{p-1}\p_xv^{\om,n}$.  Therefore, we are led to examine this term first.  Indeed, we see via the binomial theorem that we have,
\begin{align*}
(v^{\om, n})^{p-1}\p_xv^{\om,n} &= \sum_{k=0}^{p-1} \begin{pmatrix} p-1 \\ k \end{pmatrix} (\om n^{-\frac1p})^{p-1-k}(n^{-s}\cos(nx-\om^qt))^k(-n^{-s+1}\sin(nx-\omega^q t)) \\ 
&=-\sum_{k=0}^{p-1} \begin{pmatrix} p-1 \\ k \end{pmatrix} \om^{p-1-k} n^{(k+1)(\frac1p-s)}\cos^k(nx-\om^qt)\sin(nx-\omega^q t).
\end{align*}
Note that since $n>>1$, we have the dominant term out of the summation to be at $k=0$ and therefore may write
\[
(v^{\om, n})^{p-1}\p_xv^{\om,n} = -\om^{p-1}n^{\frac1p-s}\sin(nx-\omega^q t) + (\text{negligible terms}).
\]
In the first portion of our nonlocal term we also see 
\begin{align*}
au^{\om,n}+(p-a)\p_x^2u^{\om,n} = a(\omega n^{-\frac1q}+n^{-s}\cos (nx-\omega^p t))-(p-a)(n^{-s+2}\cos(nx-\omega^p t)),
\end{align*}
which after being multiplied by $(v^{\om, n})^{p-1}\p_xv^{\om,n}$ we see that
\begin{align}
(v^{\om, n})^{p-1}&\p_xv^{\om,n}(au^{\om,n}+(p-a)\p_x^2u^{\om,n}) \nonumber \\ 
&=-a\om^pn^{\frac1p-\frac1q-s}\sin(nx-\omega^q t)-a\om^{p-1}n^{\frac1p-2s}\cos(nx-\omega^p t)\sin(nx-\omega^q t) \nonumber \\ 
& \ \ \ +(p-a)\om^{p-1}n^{\frac1p-2s+2}\cos(nx-\omega^p t)\sin(nx-\omega^q t)+(\text{negligible terms}).
\end{align}
By applying the $H^\s$ norm, we find that 
\begin{equation}
\label{enl1}
\|E_{nl}^1\|_{H^\s} \lesssim n^{\frac1p-\frac1q-s+\s-2}+n^{\frac1p-2s+\s}+n^{\frac1p-2s+2}+(\text{negligible terms}).
\end{equation}

We now consider the second part of our nonlocal term $E_{nl}^2$ and find that 
\begin{align}
(v^{\om, n})^{p-1}&\p_xv^{\om,n}\p_xu^{\om,n} \nonumber \\ 
& = \left[-\om^{p-1}n^{\frac1p-s}\sin(nx-\omega^q t) + (\text{negligible terms})\right] \times \left[-n^{-s+1}\sin(nx-\omega^p t)\right] \nonumber \\ 
&=\om^{p-1}n^{\frac1p-2s+1}\sin(nx-\omega^p t)\sin(nx-\omega^q t) + (\text{negligible terms}) \nonumber \\ 
&=\frac12\om^{p-1}n^{\frac1p-2s+1}\left[\cos(\theta-\vp)-\cos(\theta+\vp)\right] + (\text{negligible terms}),
\end{align}
where $\theta = nx-\om^pt$ and $\vp = nx-\om^qt$.

Applying the $H^\s$ norm to $E_{nl}^2$ we find that
\begin{equation}
\label{enl2}
\|E_{nl}^2\|_{H^\s} \lesssim n^{\frac1p-2s+1}+n^{\frac1p-2s-1+\s}.
\end{equation}
Combining estimates \eqref{eb}, \eqref{enl1}, and \eqref{enl2} and considering that we usually measure in the $H^\s$ norm with $\s<1$, we find that
\[
\|E\|_{H^\s} \lesssim \begin{cases} \frac1p-2s+2, \ \ \ \ \ \ \ \ \ \ \ \     s<\frac1q-\s+4 \\ \frac1p-\frac1q-s+\s-2, \ \ \ s>\frac1q-\s+4. \end{cases}
\]

We now estimate the error of the second component $F$ in the $H^\s$ norm.  Let $F_B$ denote the Burgers term portion of the error $F$ and $F_{nl}$ denote the nonlocal portion.  We also note that by binomial theorem, we have 
\[
(u^{\om, n})^q = \sum_{k=0}^q \begin{pmatrix} q \\ k \end{pmatrix} (\om n^{-\frac1q})^{q-k}(n^{-s}\cos(nx-\om^pt))^k,
\]
in which the first term of this summation is $\om^qn^{-1}$.  We consider this first time multiplied by the derivative $\p_xv^{\om,n}$ and added to $\p_tv^{\om,n}$ to see that
\[
\p_tv^{\om,n}+(\om^qn^{-1})\p_xv^{\om, n} = \omega^q n^{-s}\sin(nx-\omega^q t)+\om^qn^{-1}(-n^{-s+1}\sin(nx-\omega^q t)) = 0.
\]
Therefore, in the Burgers term of the first component we are left with 
\begin{align}
F_B &= \left(\sum_{k=1}^q \begin{pmatrix} q \\ k \end{pmatrix} (\om n^{-\frac1q})^{q-k}(n^{-s}\cos(nx-\om^pt))^k\right)\left(-n^{-s+1}\sin(nx-\omega^q t)\right) \nonumber \\ 
& = -\sum_{k=1}^q \begin{pmatrix} q \\ k \end{pmatrix} \om^{q-k}n^{\frac kq-ks-s}\cos^k(nx-\om^pt)\sin(nx-\omega^q t).
\end{align}
Since we have that $n>>1$, we may just consider the first term as the others are negligible.  At $k=1$ in the summation we have that
\[
F_B = -q\om^{p-1}n^{\frac1q-2s}\cos(nx-\omega^p t)\sin(nx-\omega^q t)+(\text{negligible terms}).
\]
Using a well-known identity that
\[
\cos\theta\sin\vp = \frac12[\sin(\vp+\theta)+\sin(\vp-\theta)],
\]
we find that
\[
F_B = -\frac12\om^{p-1}n^{\frac1q-2s}[\sin(\vp+\theta)+\sin(\vp-\theta)]+(\text{negligible terms}),
\]
where
\[
\vp = nx-\om^qt \ \ \text{and} \ \ \theta = nx-\om^pt.
\]
Estimating the Burgers term $F_B$ in the $H^\s$ norm we find that
\begin{equation}
\label{fb}
\|F_B\|_{H^\s} \lesssim n^{\frac1q-2s+\s}.
\end{equation}
We now estimate the non-local term of the first component in the $H^\s$ norm.  We may break this down into two terms $F_{nl}^1$ and $F_{nl}^2$, where
\begin{align}
\label{enl12}
&F_{nl}^1 = (1-\p_x^2)^{-1}\left[(u^{\om,n})^{q-1}\p_xu^{\om,n}(bv^{\om,n}+(q-b)\p_x^2v^{\om,n})\right] \nonumber \\ 
&F_{nl}^2 = q\p_x(1-\p_x^2)^{-1}\left[(u^{\om,n})^{q-1}\p_xu^{\om,n}\p_xv^{\om,n}\right].
\end{align}
Throughout the nonlocal term we see the presence of the term $(u^{\om,n})^{q-1}\p_xu^{\om,n}$.  Therefore, we are led to examine this term first.  Indeed, we see via the binomial theorem that we have,
\begin{align*}
(u^{\om, n})^{q-1}\p_xu^{\om,n} &= \sum_{k=0}^{q-1} \begin{pmatrix} q-1 \\ k \end{pmatrix} (\om n^{-\frac1q})^{q-1-k}(n^{-s}\cos(nx-\om^pt))^k(-n^{-s+1}\sin(nx-\omega^p t)) \\ 
&=-\sum_{k=0}^{q-1} \begin{pmatrix} q-1 \\ k \end{pmatrix} \om^{q-1-k} n^{(k+1)(\frac1q-s)}\cos^k(nx-\om^pt)\sin(nx-\omega^p t).
\end{align*}
Note that since $n>>1$, we have the dominant term out of the summation to be at $k=0$ and therefore may write
\[
(u^{\om, n})^{q-1}\p_xu^{\om,n} = -\om^{q-1}n^{\frac1q-s}\sin(nx-\omega^p t) + (\text{negligible terms}).
\]
In the first portion of our nonlocal term we also see 
\begin{align*}
bv^{\om,n}+(q-b)\p_x^2v^{\om,n} = b(\omega n^{-\frac1p}+n^{-s}\cos (nx-\omega^q t))-(q-b)(n^{-s+2}\cos(nx-\omega^q t)),
\end{align*}
which after being multiplied by $(u^{\om, n})^{q-1}\p_xu^{\om,n}$ we see that
\begin{align}
(u^{\om, n})^{q-1}&\p_xu^{\om,n}(bv^{\om,n}+(q-b)\p_x^2v^{\om,n}) \nonumber \\ 
&=-b\om^qn^{\frac1q-\frac1p-s}\sin(nx-\omega^p t)-b\om^{q-1}n^{\frac1q-2s}\cos(nx-\omega^q t)\sin(nx-\omega^p t) \nonumber \\ 
& \ \ \ +(q-b)\om^{q-1}n^{\frac1q-2s+2}\cos(nx-\omega^q t)\sin(nx-\omega^p t)+(\text{negligible terms}).
\end{align}
By applying the $H^\s$ norm, we find that 
\begin{equation}
\label{fnl1}
\|F_{nl}^1\|_{H^\s} \lesssim n^{\frac1q-\frac1p-s+\s-2}+n^{\frac1q-2s+\s}+n^{\frac1q-2s+2}+(\text{negligible terms}).
\end{equation}
We now consider the second part of our nonlocal term $F_{nl}^2$ and find that 
\begin{align}
(u^{\om, n})^{q-1}&\p_xu^{\om,n}\p_xv^{\om,n} \nonumber \\ 
& = \left[-\om^{q-1}n^{\frac1q-s}\sin(nx-\omega^p t) + (\text{negligible terms})\right] \times \left[-n^{-s+1}\sin(nx-\omega^q t)\right] \nonumber \\ 
&=\om^{q-1}n^{\frac1q-2s+1}\sin(nx-\omega^p t)\sin(nx-\omega^q t) + (\text{negligible terms}) \nonumber \\ 
&=\frac12\om^{q-1}n^{\frac1q-2s+1}\left[\cos(\theta-\vp)-\cos(\theta+\vp)\right] + (\text{negligible terms}),
\end{align}
where $\theta = nx-\om^pt$ and $\vp = nx-\om^qt$.

Applying the $H^\s$ norm to $F_{nl}^2$ we find that
\begin{equation}
\label{fnl2}
\|E_{nl}^2\|_{H^\s} \lesssim n^{\frac1q-2s+1}+n^{\frac1q-2s-1+\s}.
\end{equation}
Combining estimates \eqref{fb}, \eqref{fnl1}, and \eqref{fnl2} and considering that we usually measure in the $H^\s$ norm with $\s<1$, we find that
\[
\|F\|_{H^\s} \lesssim \begin{cases} \frac1q-2s+2, \ \ \ \ \ \ \ \ \ \ \ \     s<\frac1p-\s+4 \\ \frac1q-\frac1p-s+\s-2, \ \ \ s>\frac1p-\s+4. \end{cases}
\]

\end{proof}

\textbf{Actual 2-Component genCH solutions.}  Let $(u_{\omega,n}(x,t),v_{\omega,n}(x,t))$ be the solution to the 2-Component generalized CH system i.v.p. \eqref{nl} with the initial data given by the approximate solution $(u^{\omega,n}(x,0),v^{\omega,n}(x,0))$; i.e., $(u_{\omega,n}(x,t),v_{\omega,n}(x,t))$ solves the Cauchy problem
\begin{align}
\label{apu}
&\partial_tu_{\omega,n}+v_{\omega,n}^p\partial_xu_{\omega,n}+I_1(u_{\om,n},v_{\om,n})=0 \\ 
\label{apthe}
&\partial_tv_{\omega,n}+u_{\om,n}^q\partial_xv_{\om,n}+I_2(u_{\om,n},v_{\om,n})=0 \\ 
\label{apid1}
&u_{\omega,n}(x,0)=\om n^{-\frac1q}+n^{-s}\cos(nx) \\ 
\label{apid2}
&v_{\om,n}(x,0) = \om n^{-\frac1p}+n^{-s}\cos(nx).
\end{align}
Note that \eqref{cosest} implies the initial data $(u^{\omega,n}(x,0),v^{\omega,n}(x,0)) \in H^{s} \times H^{s}$ for all $s\geq 0$, since
\begin{align*}
&\|u_{\om,n}(x,0)\|_{H^s}=\|\om n^{-\frac1q}+n^{-s}\cos(nx)\|_{H^s}\lesssim n^{-\frac1q}+1 \\ 
&\|v_{\om,n}(x,0)\|_{H^{s}}=\|\om n^{-\frac1p}+n^{-s}\cos(nx)\|_{H^{s}}\lesssim n^{-\frac1p}+1.
\end{align*}
Hence, by Theorem 1.1 in \cite{zhou3}, there exists a $T>0$ such that, for $n>>1$, the Cauchy problem \eqref{nl} has a unique solution in 
$C\left([0,T);H^s \times H^{s}\right) \cap C^1\left([0,T);H^{s-1} \times H^{s-1}\right)$ with lifespan $T>0$ such that $(u^{\om,n},v^{\om,n})$ satisfies $\|(u,v)\|_s \leq 2\|(u_0,v_0)\|_s$ for $t \in [0,T]$.

To estimate the difference between the approximate and actual solutions, we note that $w=u^{\om,n}-u_{\om,n}$ and $y=v^{\om,n}-v_{\om,n}$ satisfies the following initial value problem

\begin{align}
\label{diff1}
\p_tw = &E-yf_1\p_xu_{\om,n}-(v^{\om,n})^p\p_xw\nonumber \\ 
&-aD^{-2}\left[yg_1\p_xv_{\om,n}u_{\om,n}+(v^{\om,n})^{p-1}(w\p_xv_{\om,n}+u^{\om,n}\p_xy)\right] \nonumber \\ 
&-(p-a)D^{-2}\left[yg_1\p_xv_{\om,n}\p_x^2u_{\om,n}+(v^{\om,n})^{p-1}(\p_xy\p_x^2u_{\om,n}+\p_xv^{\om,n}\p_x^2w)\right] \nonumber \\ 
&-p\p_xD^{-2}\left[yg_1\p_xv_{\om,n}\p_xu_{\om,n}+(v^{\om,n})^{p-1}(\p_xy\p_xu_{\om,n}+\p_xv^{\om,n}\p_xw)\right],
\end{align}

\begin{align}
\label{diff2}
\p_ty = &F-wf_2\p_xv_{\om,n}-(u^{\om,n})^q\p_xy\nonumber \\ 
&-bD^{-2}\left[wg_2\p_xu_{\om,n}v_{\om,n}+(u^{\om,n})^{q-1}(y\p_xu_{\om,n}+v^{\om,n}\p_xw)\right] \nonumber \\ 
&-(q-b)D^{-2}\left[wg_2\p_xu_{\om,n}\p_x^2v_{\om,n}+(u^{\om,n})^{q-1}(\p_xw\p_x^2v_{\om,n}+\p_xu^{\om,n}\p_x^2y)\right] \nonumber \\ 
&-q\p_xD^{-2}\left[wg_2\p_xu_{\om,n}\p_xv_{\om,n}+(u^{\om,n})^{q-1}(\p_xw\p_xv_{\om,n}+\p_xu^{\om,n}\p_xy)\right],
\end{align}

where

\begin{align*}
&f_1 = \sum_{k=0}^{p-1}(v^{\om,n})^{p-1-k}v_{\om,n}^k,  \\
&g_1 = \sum_{k=0}^{p-2}(v^{\om,n})^{p-2-k}v_{\om,n}^k, \\
&f_2 = \sum_{k=0}^{q-1}(u^{\om,n})^{q-1-k}u_{\om,n}^k, \\
&g_2 = \sum_{k=0}^{q-2}(u^{\om,n})^{q-2-k}u_{\om,n}^k.
\end{align*}

\begin{lemma}
\label{diffest1}
Let $s>5/2$.  The differences $w$ and $y$ satisfy
\begin{equation}
\label{diffest}
\|(w,y)\|_\s \lesssim n^{\beta_{s,\s}},
\end{equation}
where
\begin{equation}
\label{bexp}
\beta_{s,\s} = \begin{cases} r_{s,\s}, \ \ p<q \\ j_{s,\s} \ \ \  p>q \end{cases}
\end{equation}
and $5/2<\s+1<s$, $\s<2$.  In the case where $p=q$ we have $r_{s,\s}=j_{s,\s}$ and so we may take $\beta_{s,\s}$ to be either one.
\end{lemma}

\begin{proof}
We prove by the method of energy estimates and will only do so for the first component since the second is similar.  Throughout our proof, we repeatedly make use of the Cauchy-Schwarz inequality, the algebra property for Sobolev spaces and the following two lemmas.

\begin{lemma}
\label{kp}
(Kato-Ponce)  For $s>0$ we have that
\[
\|[D^s,f]g\|_{L^2} \lesssim \|D^s\|_{L^2}\|g\|_{L^\infty}+\|\p_xf\|_{L^\infty}\|D^{s-1}g\|_{L^2}.
\]
\end{lemma}

\begin{lemma}
\label{alg2}
Let $1/2<s<1$.  Then
\[
\|fg\|_{H^{s-1}} \lesssim \|f\|_{H^s}\|g\|_{H^{s-1}}.
\]
\end{lemma}

We first apply the operator $D^\s$ to the first component \eqref{diff1}, multiply by $D^\s w$ and integrate over the torus to obtain
\begin{equation}
\frac12\frac{d}{dt}\|w\|_{H^\s}^2 = \sum_{k=0}^9I_k,
\end{equation}
where each $I_k$ represents an integral to be estimated.  We will do so below.

\textbf{Estimate for $I_1$:}  Here we have that 
\[
I_1 = \int_\ci D^\s ED^\s wdx.
\]
Applying the Cauchy-Schwarz inequality, we find that 
\begin{equation}
\label{i1est}
I_1 \lesssim \|E\|_{H^\s}\|w\|_{H^\s} \lesssim n^{r_{s,\s}}\|w\|_{H^\s}.
\end{equation}

\textbf{Estimate for $I_2$:}  Here we have that
\[
I_2 = -\int_\ci D^\s(yf_1\p_xu_{\om,n})D^\s wdx.
\]
By Cauchy-Schwarz inequality and the algebra property, we have that
\begin{equation}
\label{i2est}
I_2 \lesssim \|y\|_{H^\s}\|f_1\|_{H^\s}\|u_{\om,n}\|_{H^{\s+1}}\|w\|_{H^\s} \lesssim \|y\|_{H^\s}\|w\|_{H^\s}.
\end{equation}

\textbf{Estimate for $I_3$:}  Here we have that
\begin{align*}
I_3 &= -\int_\ci D^\s((v^{\om,n})^p\p_xw)D^\s w dx \\
&= -\left[\int_\ci[D^\s,(v^{\om,n})^p]\p_xwD^\s wdx+\int_\ci (v^{\om,n})^p\p_xD^\s wD^\s w dx\right] \\ 
&=I_{3,1}+I_{3,2}.
\end{align*}
In the first integral, we utilize the Cauchy-Schwarz inequality and the Kato-Ponce estimate from Lemma \ref{kp} to find
\begin{align*}
I_{3,1} &\lesssim \|[D^\s,(v^{\om,n})^p]\p_xw\|_{L^2}\|w\|_{H^\s} \\ 
& \lesssim (\|(v^{\om,n})^p\|_{H^\s}\|\p_xw\|_{L^\infty}+\|\p_x(v^{\om,n})^p\|_{L^\infty}\|D^{\s-1}\p_xw\|_{L^2})\|w\|_{H^\s} \\ 
& \lesssim \|w\|_{H^\s}^2.
\end{align*}
In the second integral, we apply integration by parts to obtain
\[
I_{3,2} = \frac12\int_\ci\p_x(v^{\om,n})^p(D^\s w)^2dx \lesssim \|w\|_{H^\s}^2.
\]
Combining the above estimates on $I_{3,1}$ and $I_{3,2}$, we find that
\begin{equation}
\label{i3est}
I_3 \lesssim \|w\|_{H^\s}^2.
\end{equation}

\textbf{Estimate for $I_4$:}  Here we have that
\[
I_4 = -a\int_\ci D^\s D^{-2}(yg_1\p_xv_{\om,n}u_{\om,n})D^\s wdx.
\]
We apply the Cauchy-Schwarz inequality and Lemma \ref{alg2} to find that
\begin{align}
\label{i4est}
I_4 &\lesssim \|yg_1\p_xv_{\om,n}u_{\om,n}\|_{H^{\s-2}}\|w\|_{H^\s} \nonumber \\ 
&\lesssim \|yg_1u_{\om,n}\|_{H^{\s-1}}\|\p_xv_{\om,n}\|_{H^{\s-2}}\|w\|_{H^\s} \nonumber \\ 
& \lesssim \|y\|_{H^\s}\|w\|_{H^\s}.
\end{align}

\textbf{Estimate for $I_5$:}  Here we have that
\[
I_5 = -a\int_\ci D^\s D^{-2}[(v^{\om,n})^{p-1}(w\p_xv_{\om,n}+u^{\om,n}\p_xy)]D^\s wdx.
\]
Another application of the Cauchy-Schwarz inequality along with Lemma \ref{alg2}, we find that
\begin{equation}
\label{i5est}
I_5 \lesssim \|w\|_{H^\s}^2+\|w\|_{H^\s}\|y\|_{H^\s}.
\end{equation}

\textbf{Estimate for $I_6$:}  Here we have that
\[
I_6=-(p-a)\int_\ci D^\s D^{-2}(yg_1\p_xv_{\om,n}\p_x^2u_{\om,n})D^\s wdx.
\]
Another application of the Cauchy-Schwarz inequality and Lemma \ref{alg2} yields
\begin{equation}
\label{i6est}
I_6 \lesssim \|w\|_{H^\s}\|y\|_{H^\s}.
\end{equation}

\textbf{Estimate for $I_7$:}  Here we have that
\[
I_7 = -(p-a)\int_\ci D^\s D^{-2}[(v^{\om,n})^{p-1}(\p_xy\p_x^2u_{\om,n}+\p_xv^{\om,n}\p_x^2w)]D^\s wdx.
\]
Another application of the Cauchy-Schwarz inequality and Lemma \ref{alg2} yields
\begin{equation}
\label{i7est}
I_7 \lesssim \|w\|_{H^\s}^2+\|w\|_{H^\s}\|y\|_{H^\s}.
\end{equation}

\textbf{Estimate for $I_8$:}  Here we have that
\[
I_8 = -p\int_\ci D^\s\p_xD^{-2}(yg_1\p_xv_{\om,n}\p_xu_{\om,n})D^\s wdx.
\]
An application of the Cauchy-Schwarz inequality along with the algebra property for Sobolev spaces yields
\begin{equation}
\label{i8est}
I_8 \lesssim \|w\|_{H^\s}\|y\|_{H^\s}.
\end{equation}

\textbf{Estimate for $I_9$:}  Here we have that
\[
I_9 = -p\int_\ci D^\s\p_xD^{-2}[(v^{\om,n})^{p-1}(\p_xy\p_xu_{\om,n}+\p_xv^{\om,n}\p_xw)]D^\s wdx.
\]
An application of the Cauchy-Schwarz inequality along with the algebra property for Sobolev spaces yields
\begin{equation}
\label{i9est}
I_9 \lesssim \|w\|_{H^\s}^2+\|w\|_{H^\s}\|y\|_{H^\s}.
\end{equation}

Collecting estimates for $I_1$ through $I_9$ we find that
\begin{equation}
\label{wen1}
\frac12\frac{d}{dt}\|w\|_{H^\s}^2 \lesssim \|w\|_{H^\s}^2+\|y\|_{H^\s}\|w\|_{H^\s}+n^{r_{s,\s}}\|w\|_{H^\s}
\end{equation}
or equivalently
\begin{equation}
\label{wen2}
\frac{d}{dt}\|w\|_{H^\s} \lesssim \|w\|_{H^\s}+\|y\|_{H^\s}+n^{r_{s,\s}}.
\end{equation}
Applying the same methods to the second component, we find that
\begin{equation}
\label{yeqn1}
\frac{d}{dt}\|y\|_{H^\s} \lesssim \|w\|_{H^\s}+\|y\|_{H^\s}+n^{j_{s,\s}}.
\end{equation}

We now wish to solve \eqref{wen2} and \eqref{yeqn1} simultaneously.  Indeed, by adding the two, we find that
\begin{equation}
\frac{d}{dt}\|(w,y)\|_\s \lesssim \|(w,y)\|_\s +n^{r_{s,\s}}+n^{j_{s,\s}}.
\end{equation}
Upon solving this differential equation, we obtain the desired result.

\end{proof}

\textbf{Proof of Non-Uniform Dependence.}  Consider the sequences of (actual) unique solutions to the Cauchy problem \eqref{nl} given by $(u_{1,n}(x,t),v_{1,n}(x,t))$ and $(u_{-1,n}(x,t),v_{-1,n}(x,t))$ with initial data $$(u_{1,n}(x,0),v_{1,n}(x,0)) \  \text{and} \  (u_{-1,n}(x,0),v_{-1,n}(x,0))$$ respectively.  Applying Lemma \ref{diffest1} we have
\begin{equation*}
\|(u^{\om,n}(t)-u_{\om,n}(t),v^{\om,n}(t)-v_{\om,n}(t))\|_\s  \lesssim n^{\beta_{s,\s}}, \ \ 0\leq t \leq T.
\end{equation*}
Furthermore, from our solution size estimate, we have that 
\begin{equation*}
\|(u_{\om,n}(t),v_{\om,n}(t))\|_k \lesssim \|(u_{\om,n}(0),v_{\om,n}(0))\|_k \lesssim n^{-\frac1p}+n^{-\frac1q}+n^{k-s}
\end{equation*}
for $k>s$.

This implies that
\begin{align*}
&\|u_{\om,n}(t)\|_{H^k} \lesssim \|u_{\om,n}(0)\|_{H^k}+\|v_{\om,n}(0)\|_{H^{k}}, \\
&\|v_{\om,n}(t)\|_{H^{k}} \lesssim \|u_{\om,n}(0)\|_{H^k}+\|v_{\om,n}(0)\|_{H^{k}}.
\end{align*}
Thus, by utilizing the triangle inequality, we obtain
\begin{align*}
&\|u^{\om,n}(t)-u_{\om,n}(t)\|_{H^k} \lesssim n^{-\frac1p}+n^{-\frac1q}+n^{k-s} \lesssim n^{k-s}, \\ 
&\|v^{\om,n}(t)-v_{\om,n}(t)\|_{H^{k}} \lesssim n^{-\frac1p}+n^{-\frac1q}+n^{k-s} \lesssim n^{k-s}.
\end{align*}
Now we wish to use the following interpolation lemma.
\vspace{.1in}
\begin{lemma}
\label{hsinterp}
Let $f \in H^s$ and $s_1<s<s_2$.  Then
\begin{equation*}
\|f\|_{H^s} \leq \|f\|_{H^{s_1}}^{\frac{s_2-s}{s_2-s_1}}\|f\|_{H^{s_2}}^{\frac{s-s_1}{s_2-s_1}}.
\end{equation*}
\end{lemma}
Let $w(t)=u^{\om,n}(t)-u_{\om,n}(t)$ and $y(t)=v^{\om,n}(t)-v_{\om,n}(t)$.  Interpolating between the $H^\s$ norm and the $H^k$ norm to get an $H^s$ estimate for $w(t)$ and $y(t)$ where $k=[s]+2$
yields the following estimate:
\begin{align*}
&\|w\|_{H^s} \leq \|w\|_{H^\s}^{\frac{k-s}{k-\s}}\|w\|_{H^k}^{\frac{s-\s}{k-\s}} \lesssim \left(n^{\beta_{s,\s}}\right)^\frac{k-s}{k-\s}\left(n^{k-s}\right)^\frac{s-\s}{k-\s} \simeq n^{\alpha_{s,\s}} \\ 
&\|y\|_{H^{s-1}} \leq \|y\|_{H^{\s}}^\frac{k-s}{k-\s}\|y\|_{H^{k}}^\frac{s-\s}{k-\s} \lesssim \left(n^{\beta_{s,\s}}\right)^\frac{k-s}{k-\s}\left(n^{k-s}\right)^\frac{s-\s}{k-\s} \simeq n^{\alpha_{s,\s}}.
\end{align*}
where
\[
\alpha_{s,\s} = \left(\frac{k-s}{k-\s}\right)(\beta_{s,\s}+s-\s)<0.
\]
At time $t=0$, we have
\begin{align*}
&\|u_{1,n}(0)-u_{-1,n}(0)\|_{H^s}=\|2n^{-\frac1q}\|_{H^s} \simeq n^{-\frac1q} \rightarrow 0 \ \ \text{as} \ \ n \rightarrow \infty \\ 
&\|v_{1,n}(0)-v_{-1,n}(0)\|_{H^{s}}=\|2n^{-\frac1p}\|_{H^{s}} \simeq n^{-\frac1p} \rightarrow 0 \ \ \text{as} \ \ n \rightarrow \infty.
\end{align*}
However, at time $t>0$ we have
\begin{align}
\label{udiff1}
\|u_{1,n}(t)-u_{-1,n}(t)\|_{H^s} &\geq \|u^{1,n}(t)-u^{-1,n}(t)\|_{H^s}-\|u^{1,n}(t)-u_{1,n}(t)\|_{H^s} \nonumber \\ &-\|u^{-1,n}(t)-u_{-1,n}(t)\|_{H^s}.
\end{align}
and
\begin{align}
\label{vdiff1}
\|v_{1,n}(t)-v_{-1,n}(t)\|_{H^{s}} &\geq \|v^{1,n}(t)-v^{-1,n}(t)\|_{H^{s}}-\|v^{1,n}(t)-v_{1,n}(t)\|_{H^{s}} \nonumber \\ &-\|v^{-1,n}(t)-v_{-1,n}(t)\|_{H^{s}}.
\end{align}
Adding \eqref{udiff1} and \eqref{vdiff1} we find that
\begin{align}
\|(u_{1,n}(t)-u_{-1,n}(t),v_{1,n}(t)-\rho_{-1,n}(t))\|_s &\geq \|u^{1,n}(t)-u^{-1,n}(t)\|_{H^s}-cn^{\alpha_{s,\s}} \nonumber \\ 
&+\|v^{1,n}(t)-v^{-1,n}(t)\|_{H^{s}}-an^{\alpha_{s,\s}},
\end{align}
where $a$ and $c$ are constants.
Taking the limit infimum of both sides gives us
\begin{align*}
\lim_{n \rightarrow \infty} \inf &\left(\|(u_{1,n}(t)-u_{-1,n}(t),v_{1,n}(t)-v_{-1,n}(t))\|_s \right) \nonumber \\ 
&\geq \lim_{n \rightarrow \infty} \inf \left(\|(u^{1,n}(t)-u^{-1,n}(t),v^{1,n}(t)-v^{-1,n}(t))\|_s \right).
\end{align*}
Hence, to finish the argument for Theorem \ref{nonunif}, we need only find a lower bound for the difference of known approximate solutions:
\begin{align*}
&u^{1,n}(t)-u^{-1,n}(t)=2n^{-\frac1q}+n^{-s}[\cos(nx-t)-\cos(nx-(-1)^pt)], \\
&v^{1,n}(t)-v^{-1,n}(t)=2n^{-\frac1p}+n^{-s}[\cos(nx-t)-\cos(nx-(-1)^qt)].
\end{align*}

Using the identity $\cos \alpha - \cos \beta=-2\sin\frac{\alpha+\beta}{2}\sin\frac{\alpha-\beta}{2}$ yields
\begin{align*}
&u^{1,n}(t)-u^{-1,n}(t)=2n^{-\frac1q}-2n^{-s}\sin\left(\frac{2nx-t(1+(-1)^p)}{2}\right)\sin\left(\frac{(-1)^p-1}{2}t\right), \\
&v^{1,n}(t)-v^{-1,n}(t)=2n^{-\frac1p}-2n^{-s}\sin\left(\frac{2nx-t(1+(-1)^q)}{2}\right)\sin\left(\frac{(-1)^q-1}{2}t\right).
\end{align*}
Therefore,
\begin{equation*}
\|(u^{1,n}(t)-u^{-1,n}(t),v^{1,n}(t)-v^{-1,n}(t))\|_s \gtrsim |\sin t|-n^{-\frac1p}-n^{-\frac1q},
\end{equation*}
when $p,q$ are both odd or they are mixed even/odd.  In the case when they are both even, choose $\om=1,0$ and replace $|\sin t|$ with $|\sin(t/2)|$.

Taking the limit infimum of both sides yields 
\begin{equation*}
\lim_{n \rightarrow \infty} \inf \left(\|(u^{1,n}(t)-u^{-1,n}(t),v^{1,n}(t)-v^{-1,n}(t))\|_s \right) \gtrsim |\sin(t)|,
\end{equation*}
which concludes our proof of Theorem \ref{nonunif}.    $\Box$

\end{document}